\documentclass [12pt,reqno]{amsart}
\usepackage{times}
\usepackage{mathtools}
\usepackage[toc,page]{appendix} 
\usepackage{bm}
\usepackage{pifont}
\numberwithin{equation}{subsection}

\usepackage{amssymb,latexsym,amsmath,amsfonts, eucal}
\usepackage[usenames,dvipsnames,svgnames,table]{xcolor}

\newtheorem{thm}{Theorem}[section]
      \newtheorem{lemma}[thm]{Lemma}

      \newtheorem{example}[thm]{Example}
      
      \newtheorem{rmk}[thm]{Remark}
      
      \numberwithin{equation}{section}
\title [ vector valued de Branges spaces]{ J-contractive operator valued functions, vector valued de Branges spaces and functional models}

\author[Garg]{Bharti Garg}

\address{
	Department of Mathematics\\
	Indian Institute of Technology Ropar\\
	140001\\
	India}

\email{ {bharti.20maz0012@iitrpr.ac.in},{bhartigargfdk@gmail.com}}

\author[Sarkar]{Santanu Sarkar}

\address{
	Department of Mathematics\\
	Indian Institute of Technology Ropar\\
	140001\\
	India}

\email{ {santanu@iitrpr.ac.in},{ santanu87@gmail.com}}

\begin{document}


\begin{abstract}
The aim of this paper is to study the vector valued de Branges spaces, which are based on $J$-contractive operator valued analytic functions, and to explore their role in the functional models for simple, closed, densely defined, symmetric operators with infinite deficiency indices.
\end{abstract}
\maketitle

\tableofcontents
\section{Introduction}
\label{Section 1}
The origin of \( J \)-theory, where \( J \) is a unitary and self-adjoint operator on a Hilbert space \(\mathfrak{X}\), can be traced back to Pontryagin's seminal article \cite{Pontryagin1}, which was influenced by Sobolev's research work \cite{Sobolev1} involving mechanical problems. This foundational work was subsequently expanded by authors such as Krein, Iokhvidov, Langer, and Bognar. Potapov's paper \cite{Potapov1} on the finite-dimensional analytic aspects of the theory eventually led Ginzburg \cite{Ginzburg1} to develop its infinite-dimensional counterpart.
The study of \( J \)-theory is crucial in addressing a variety of problems across several fields, including mathematical systems and networks, control theory, stochastic processes, operator theory, and classical analysis. Notably, it is a key component in the analysis of both direct and inverse problems for canonical systems of integral and differential equations.

A bounded linear operator \( U \) on \(\mathfrak{X}\) is said to be \( J \)-contractive if \( U^*JU \preceq J \). The theory of \( J \)-contractive and \( J \)-inner matrix valued functions has been extensively studied by Potapov \cite{Potapov1}, obtaining fundamental results regarding multiplicative representation of $J$-contractive matrix valued analytic functions. These results play a crucial role in the spectral theory of non self-adjoint operators in Hilbert spaces. A comprehensive study of vector valued reproducing kernel Hilbert spaces (RKHS), derived from a class of \( J \)-contractive and \( J \)-inner matrix valued functions, is presented in the monograph by Arov and Dym \cite{ArD08}. Recent works by Dym (see \cite{DymJFA}, \cite{ArDymone}), investigate and establish connections between two classes of vector valued RKHS, known as de Branges spaces, that exist in the literature: \(\mathcal{B}(\mathfrak{E})\) and \( \mathcal{H}(U) \) spaces. These spaces were first introduced by de Branges and applied to various analytical problems, including inverse problems for canonical differential systems. Motivated from the scalar valued Paley-Wiener spaces, de Branges explored the spaces of scalar valued entire functions which originates with \cite{brange} and a comprehensive study of these spaces can be found in his book \cite{Brange}. These spaces were further generalized in the setting of matrix and operator valued entire and meromorphic functions in his works (see \cite{Branges 1}, \cite{Branges 2}, \cite{Branges 3}). Recently, in \cite{mahapatra}, the \( \mathcal{B}(\mathfrak{E}) \) spaces based on a pair of Fredholm operator valued entire functions, referred to as de Branges operator, have been investigated. These spaces serve as a generalization of vector valued Paley-Wiener spaces and also serve as functional models for a Krein's class of entire operators with infinite deficiency indices. Connections between quasi-Lagrange type interpolation and \( \mathcal{B}(\mathfrak{E}) \) spaces have been found recently in \cite{mahapatra3}. In this paper, we aim to investigate the vector valued de Branges spaces $\mathcal{H}(U)$ that are based on $J$-contractive operator valued analytic functions. Alpay and Dym in \cite{alpay} established the connections between Krein's theory of symmetric operators on Hilbert spaces with finite and equal deficiency indices, and the de Branges spaces $\mathcal{H}(U)$ which are based on matrix valued reproducing kernels. The paper seeks to investigate how the de Branges spaces based on $J$-contractive operator valued analytic functions relate to closed, densely defined, simple, symmetric operators with infinite deficiency indices.

The paper is divided into the following sections. An overview and brief history, as well as the paper's layout and notations, are provided in Section $\ref{Section 1}$. Section $\ref{Section 2}$ contains some examples and preliminary information about RKHS and vector valued de Branges spaces based on operator valued reproducing kernels. Section $\ref{Section 3}$ outlines how the de Branges spaces serve as a functional model for simple, closed, densely defined, symmetric operators with infinite deficiency indices.
\\
The following notations will be used throughout the paper: 
\begin{itemize} 
\item $\mathbb{C},$ $\mathbb{C}_+$, and $\mathbb{C}_-$ denote the complex plane, the open upper half plane, and the open lower half plane, respectively. 
\item  $I$ denotes the identity operator on some Hilbert space.
\item $\mathfrak{X}$ denotes a complex separable infinite dimensional Hilbert space.
\item $B(\mathfrak{X})$ denotes the space of all bounded linear operators on Hilbert space $\mathfrak{X}$.
\item $S_{B(\mathfrak{X})}$ denotes the space of all contractive linear operators on $\mathfrak{X}.$
\item $H^{\infty}_{B(\mathfrak{X})}(\mathbb{C}_+)= \{ f: \mathbb{C}_+ \rightarrow B(\mathfrak{X}) \hspace{0.1cm}\vert\hspace{0.1cm} f~ \mbox{is bounded and holomorphic} \}$.
\item $S_{B(\mathfrak{X})}(\mathbb{C}_+) = \bigl \{ S \in H^{\infty}_{B(\mathfrak{X})}(\mathbb{C}_+) : S(z) \in S_{B(\mathfrak{X})} ~\mbox{for all}~ z \in \mathbb{C}_+ \bigl \} $. 
\item $\rho_\xi(z)= -2 \pi i (z-\bar{\xi})$.
\item For an operator $A$,
\begin{itemize}
\item [(i)]$A^*$ denotes the adjoint operator.
\item [(ii)]$A\succeq 0$ denotes that $A$ is positive semi-definite.
\item [(iii)] $\mathrm{rng} (A)$, $\ker (A)$, and $\mathcal{D}(A)$  denote the range, kernel, and domain of $A$, respectively.
\item [(iv)] A point $\alpha$ is said to be a point of regular type for $A$ if there exists a positive constant $c_\alpha$ such that $$\Vert (A-\alpha I )g\Vert \geq c_\alpha \Vert g \Vert \hspace{0.3cm} \text{for all} \hspace{0.2cm} g \in \mathcal{D}(A).$$
\end{itemize}
\item $R_z$ denotes the generalized backward shift operator of $\mathfrak{X}$ valued functions and is defined by \begin{equation}
    (R_z g)(\xi) := \left\{
   \begin{array}{ll}
         \frac{g(\xi)-g(z)}{\xi-z}  & \mbox{if }~ \xi \neq z \vspace{0.1cm} \\ 
        g'(z) & \mbox{if }~ \xi = z
   \end{array} \right.
\end{equation}
for every $z,~ \xi \in \mathbb{C}$.
\end{itemize}
 
\section{Preliminaries to RKHS and vector valued de Branges spaces}
\label{Section 2}
The theory of RKHS was introduced at the beginning of the 20th century. The groundwork for RKHS was laid by Aronszajn in 1950. In this section, we briefly recall the definition of vector valued RKHS and of positive kernel functions. In this section, we also recall some relevant examples of RKHS. For a detailed study on RKHS, see \cite{Aronszajn 1950} and \cite{Paulsen}.

A Hilbert space $\mathcal{H}$ of $\mathfrak{X}$ valued functions defined on a nonempty set $\Omega \subseteq \mathbb{C}$ is said to be a RKHS if there exists a $B(\mathfrak{X})$ valued function $K_{\omega}(\lambda)$ defined on $\Omega \times \Omega$ such that the following holds:
\begin{itemize}
\item $K_{\omega}u \in \mathcal{H}$ for all $u \in \mathfrak{X}$ and $\omega \in \Omega$.
\item $\left\langle f, K_{\omega}u \right\rangle_\mathcal{H} = \langle f(\omega),u \rangle_\mathfrak{X}$ for all $f \in \mathcal{H},~ u \in \mathfrak{X},~ \omega \in \Omega$.
\end{itemize}

The \( B(\mathfrak{X}) \)-valued function \( K_{\omega}(\lambda) \) is unique and is called the reproducing kernel (RK) of the Hilbert space \( \mathcal{H} \). Let \( \delta_{\omega} \), given by \( \delta_{\omega}(f) = f(\omega) \), denote the point evaluation operator on \( \mathcal{H} \) at the point \( \omega \). Then \( K_{\omega}(\lambda) \) can be written as \( K_{\omega}(\lambda) = \delta_{\lambda} \delta_{\omega}^* \). A function \( K : \Omega \times \Omega \rightarrow B(\mathfrak{X}) \) is said to be a positive kernel if, for any choice of \( n \in \mathbb{N} \), \( \omega_{1}, \ldots, \omega_n \in \Omega \), and \( u_1, \ldots, u_n \in \mathfrak{X} \),
\begin{equation*}
\sum_{i,j=1}^{n} \left\langle K_{\omega_j} (\omega_i) u_j , u_i \right\rangle \geqslant 0.
\end{equation*}

It can be readily verified that the RK of an RKHS is a positive kernel function. Furthermore, the vector valued version of Moore's theorem \cite[Theorem 6.12]{Paulsen} asserts that for any positive kernel \(K\), there exists a unique RKHS \(\mathcal{H}\) of vector-valued functions such that \(K\) serves as the reproducing kernel of \(\mathcal{H}\).

The following lemma from \cite[Lemma 5.6]{ArD08}, gives a useful characterization of RKHS of vector valued holomorphic functions.
\begin{lemma}\label{Lemma 2.1}
Let $\mathcal{H}$ be an RKHS of $\mathfrak{X}$ valued functions defined on some nonempty open set $\Omega \subseteq \mathbb{C}$ with RK $K_{\omega}(\lambda)$. Then $\mathcal{H}$ is an RKHS of $\mathfrak{X}$ valued holomorphic functions if $K_{\omega}(\lambda)$ is a holomorphic function of $\lambda \in \Omega$ for every $\omega \in \Omega$ and $K_{\omega}(\omega)$ is continuous on every compact subset of $\Omega$.
\end{lemma}

Next, we give some relevant examples of vector valued RKHS with operator valued kernel functions.
\begin{example}
It is well known that the Hardy space over the upper half plane given as 
\[ H^2_{\mathfrak{X}}(\mathbb{C}_+) = \left\{ f: \mathbb{C}_+ \rightarrow \mathfrak{X} : f \text{ is holomorphic and } \sup_{y>0} \int_{-\infty}^{\infty} \Vert f(x+iy) \Vert^2 dx < \infty \right\} \]
is an RKHS of \( \mathfrak{X} \) valued holomorphic functions with RK
\[ K_{\omega}(\lambda) = \frac{I_{\mathfrak{X}}}{\rho_{\omega}(\lambda)}. \]
\end{example}
\begin{example}
Let \( S \in S_{B(\mathfrak{X})}(\mathbb{C}_+) \) be an operator valued analytic function and corresponding to this \( S \), consider a function defined by 
\[ \Gamma_{\xi}^{S}(\lambda) = \frac{I-S(\lambda)S(\xi)^*}{\rho_{\xi}(\lambda)}. \]
It is a positive operator valued kernel function on \( \mathbb{C}_+ \times \mathbb{C}_+ \) which guarantees the existence of a unique RKHS of analytic vector valued functions denoted as \( \mathcal{H}(S) \). These spaces are known as the de Branges-Rovnyak spaces corresponding to the upper half plane.
\end{example}
Next, we recall the definition of Fredholm operators that will be used in the description of another class of vector valued de Branges space \(\mathcal{B}(\mathfrak{E})\). An operator \( A \in  B(\mathfrak{X}) \) is said to be Fredholm if it satisfies the following conditions:
\begin{itemize}
\item $\mathrm{rng}  (A) \) is closed in \( \mathfrak{X} \).
\item \( \dim(\ker(A)) < \infty \).
\item \( \dim(\ker(A^*)) < \infty \).
\end{itemize} 
The index of a Fredholm operator is defined by  
\begin{equation}
\text{index}(A) = \dim(\ker(A)) - \dim(\ker(A^*)).
\end{equation}

\begin{thm} \label{Theorem 2.5}
If \( A,B \in B(\mathfrak{X}) \). Then the following assertions are true.
\begin{itemize}
\item[(1)] \( A \) is a Fredholm operator iff \( A^* \) is a Fredholm operator and \[ \text{index}(A) = - \text{index}(A^*). \]
\item[(2)] If \( A \) is a Fredholm operator then \( A \) is invertible iff \[ \text{index}(A) = 0 \quad \text{and} \quad \ker(A) \ (\text{or} \ \ker(A^*)) = 0. \]
\end{itemize}
\end{thm}
Now, we discuss two classes of vector valued de Branges spaces. Here we consider \(\Omega \subseteq \mathbb{C}\), a nonempty open set which is symmetric about the real axis.

Let \(U(\lambda)\) be an analytic operator valued function defined on \(\Omega\) that satisfies the following identity:
\begin{equation}\label{Equation 2.2}
U(\lambda)JU(\bar{\lambda})^* = J = U(\bar{\lambda})^*JU(\lambda)~ \mbox{for all}~ \lambda \in \Omega.
\end{equation}
The above equality implies that $U(\lambda)$ is an invertible operator for all $\lambda \in \Omega$ where $$U(\lambda)^{-1}=JU(\bar{\lambda})^*J. $$
The class of vector valued de Branges spaces, denoted by \(\mathcal{H}(U)\), is the RKHS of analytic vector valued functions defined on \(\Omega\) with the kernel function given by \begin{equation} 
K_\xi^{U}(\lambda):= \left\{
    \begin{array}{ll}
         \frac{J-U(\lambda)J U(\xi)^*}{\rho_\xi(\lambda)}  & \mbox{if } \lambda,~\xi \in \Omega~\mbox{and}~\lambda \neq \overline{\xi} \vspace{0.1cm}\\
           \frac{U'(\bar{\xi}) J U(\xi)^*}{2\pi i} & \mbox{if }\xi \in \Omega~\mbox{and}~ \lambda =\overline{\xi}                                                        . \end{array} \right.\label{Equation 2.3}
\end{equation}
Note that the positivity of the above kernel implies that $U(\lambda)^*$ is $J$-contractive on $\Omega \cap \mathbb{C}_+$, $(-J)$-contractive on $\Omega \cap \mathbb{C}_-$ and $U(\lambda)$ is $J$-unitary (i.e., $U(\lambda)JU(\lambda)^*=J$ and $U(\lambda)^*JU(\lambda)=J$) on $\Omega \cap \mathbb{R}$. 

We recall another class of de Branges spaces, denoted by $\mathcal{B}(\mathfrak{E})$, which was constructed in \cite{mahapatra} when $\Omega =\mathbb{C}$. This class corresponds to a de Branges operator, which is a pair of operator valued analytic functions denoted by $\mathfrak{E}(z)=(E_-(z),~ E_+(z))$, such that the components of $\mathfrak{E}(z)$ satisfy the following conditions:
\begin{itemize}
\item[(1)] $E_+$, $E_- : \Omega \rightarrow B(\mathfrak{X})$ both are Fredholm operators for all $z \in \Omega$.
\item[(2)] $E_+$ and $E_- $ both are invertible for at least one point in $\Omega$.
\item[(3)] $E_+^{-1}E_-$ satisfy the following conditions:
\begin{eqnarray*}(E_+^{-1}E_-)^*(\lambda)(E_+^{-1}E_-)(\lambda) &\preceq & I ~\mbox{for all}~ \lambda \in \Omega\cap \mathbb{C}_+, \\
(E_+^{-1}E_-)^*(\lambda)(E_+^{-1}E_-)(\lambda) &=& I ~\mbox{for all}~ \lambda \in \Omega\cap \mathbb{R},\\
(E_+^{-1}E_-)(\lambda)(E_+^{-1}E_-)^*(\lambda)&=& I ~\mbox{for all}~ \lambda \in \Omega\cap \mathbb{R}.
\end{eqnarray*}
\end{itemize}
Now corresponding to a de Branges operator, we consider the kernel given by: \begin{equation} \label{Equation 2.4}
K_\xi^{\mathfrak{E}}(z):= \left\{
    \begin{array}{ll}
         \frac{E_+(z)E_+(\xi)^*-E_-(z)E_-(\xi)^*}{\rho_\xi(z)}  & \mbox{if } z,~\xi \in \Omega~\mbox{and}~z \neq \overline{\xi} \vspace{0.2cm}\\
         \frac{E_+^{'} (\overline{\xi})E_+(\xi)^*- E_-^{'}(\overline{\xi})E_-(\xi)^*}{-2\pi i} & \mbox{if }\xi \in \Omega~\mbox{and}~ z = \overline{\xi}.
    \end{array} \right.
\end{equation} 
The kernel is positive on  $\Omega \times \Omega$, and $\mathcal{B}(\mathfrak{E})$ denotes the corresponding unique RKHS of $\mathfrak{X}$-valued analytic functions on $\Omega$.

\section{$\mathcal{H}_0(U)$ as a model space}
\label{Section 3} 
In this section, we focus on a subclass of the \(\mathcal{H}(U)\) spaces for which $I-U(\lambda)$ is not compact for all $\lambda \in \Omega$. This class will be denoted by \(\mathcal{H}_0(U)\).
Here, we aim to use \(\mathcal{H}_0(U)\) as a model space for closed, densely defined, simple, symmetric operators with infinite deficiency indices. To the best of our knowledge, corresponding to the infinite deficiency indices case, this is the first functional model for such operators. This class of operators with equal and finite deficiency indices was introduced by Krein in \cite{fundamentalkrein} and its connections with $\mathcal{H}(U)$ and $\mathcal{B}(\mathfrak{E})$ spaces corresponding to matrix valued RK's were studied in \cite{alpay}, \cite{sarkar}. This framework was further extended in \cite{mahapatra} to encompass the case of Krein's entire operators with infinite deficiency indices and establish its connections with the de Branges space \(\mathcal{B}(\mathfrak{E})\) which are based on a pair of Fredholm operator valued entire functions.

Before proceeding, we recall the following known characterization of de Branges spaces that will be required in this section. This characterization is originally due to de Branges \cite{Branges 1} and later a technical improvement of this characterization was done by Rovnyak in \cite{Rovnyak}. Here, we present Rovnyak's characterization for $\mathcal{H}(U)$ spaces corresponding to a nonempty open subset $\Omega \subset \mathbb{C}$ which is symmetric about the real axis.

\begin{thm}\label{Theorem 3.1}
Let $\mathcal{H}$ be an RKHS of vector valued analytic functions defined on $\Omega$ which is symmetric about the real axis. Then $\mathcal{H}$ is a $\mathcal{H}(U)$ space if and only if the following conditions hold:
\begin{enumerate}
\item $\mathcal{H}$ is $R_{\alpha}$ invariant for all $\alpha \in \Omega$.
\item The de Branges identity given by  \begin{equation}
\langle R_{\alpha}f,g \rangle -\langle f,R_{\beta}g \rangle -(\alpha-\bar{\beta})\langle R_{\alpha}f, R_{\beta}g \rangle= 2\pi i g(\beta)^* Jf(\alpha),
\end{equation}
holds true for all $f,~g \in \mathcal{H}$ and $\alpha,~\beta \in \Omega$.
\end{enumerate}  
\end{thm}

Let \(E\) be a simple, closed, densely defined, symmetric operator on the Hilbert space \(\mathfrak{X}\) with infinite deficiency indices, and let $Y$ be a closed infinite dimensional subspace of $\mathfrak{X}$. Denote \(\mathfrak{M}_\xi\) as the range of the operator \(E-\xi I\) for any \(\xi \in \mathbb{C}\). Following \cite[Appendix 1]{kreinlect}, a point $\xi \in \mathbb{C}$ is called a $Y$-regular point for the operator $E$ if $\mathfrak{M}_{\xi}$ is a closed subspace of $\mathfrak{X}$ and the following direct sum decomposition holds true: \begin{equation} \label{Equation 3.2}
\mathfrak{X}= \mathfrak{M}_\xi \dotplus Y .
\end{equation} 
It is known that every $Y$-regular point is a point of regular type of the operator $E$ and the collection of all $Y$-regular points forms a open subset of $\mathbb{C}$. We will consider those non empty open subsets $\Omega \subset \mathbb{C}$ which are symmetric about the real axis and consist of $Y$-regular points of $E$. If $\Omega= \mathbb{C}$, then $E$ is called an entire operator or $Y$-entire operator.
\begin{rmk}\label{Remark 3.2}
Given a simple, closed, densely defined, symmetric operator $E$ with infinite deficiency indices, and assuming the existence of a point of regular type (denoted by $a$) on the real line; we can always establish the existence of the following: 
\begin{itemize}
\item A fixed infinite dimensional closed subspace $Y$ of $\mathfrak{X}$; and 
\item A non empty open subset $\Omega \subset \mathbb{C}$, which is symmetric about the real axis and consists of $Y$-regular points of $E$.
\end{itemize}
It is known that \(\mathfrak{M}_\xi\) is a closed subspace of \(\mathfrak{X}\) for every \(\xi \in \mathbb{C}_+ \cup \mathbb{C}_-\). 
 Consider the Shtrauss extension $\tilde{E_a}$ of $E$ given by \begin{equation}\linebreak
\mathcal{D}(\tilde{E_a})= \mathcal{D}(E) \dotplus \mathfrak{M}_a^{\perp}, \hspace{0.1cm} \tilde{E_a}(f_E+ \phi_a)=Ef_E+a\phi_a,
\end{equation}where, $ f_E \in \mathcal{D}(E), \phi_a \in \mathfrak{M}_a^{\perp}$.
This extension $\tilde{E_a}$ yields a self-adjoint operator. For more details, we refer to \cite{straus}, \cite{arlinski}. Thus, for any $z \in  \mathbb{C}_+ \cup \mathbb{C_-}$ and $f \in \mathfrak{X},$ there exists a unique $f_E \in \mathcal{D}(E)$ and $\phi_a \in \mathfrak{M}_a^{\perp}$ such that \begin{equation}
f=(\tilde{E_a}-zI)(f_E+\phi_a)= (E-zI)f_E+(a-z)\phi_a.
\end{equation}
We obtain the following direct sum decomposition: \begin{equation} \label{Equation 3.5}
\mathfrak{X}= \mathfrak{M}_z \dotplus \mathfrak{M}_a^{\perp}\hspace{0.5cm} \text{for all} \hspace{0.1cm} z \in \mathbb{C}_+ \cup \mathbb{C_-},
\end{equation}
with the intersection property: \begin{equation}
\mathfrak{M}_z \cap \mathfrak{M}_a^{\perp}= \{ 0 \} \hspace{0.5cm} \text{for all} \hspace{0.1cm} z \in \mathbb{C}_+ \cup \mathbb{C_-}.
\end{equation}
Clearly, $\mathfrak{M}_a^{\perp}$ is a closed subspace. Since, $a$ is a point of regular type of operator $E$, there exists an interval $(a- \epsilon, a+ \epsilon)$ consisting of points of regular type of $E$. Since $E$ is symmetric, as a result, $\mathbb{C}_+ \cup \mathbb{C_-} \cup (a- \epsilon, a+ \epsilon)$ becomes a connected subset of the field of regularity of $E$. Moreover, $E$ is an operator with infinite deficiency indices and it follows from  \cite[Section 78]{Akhiezer}, $\mathfrak{M}_a^{\perp}$ is infinite dimensional. Therefore, the required infinite dimensional closed subspace $Y$ can be choosen to be $\mathfrak{M}_a^{\perp}$. Furthermore, from \eqref{Equation 3.5}, \begin{equation} \label{Equation 3.7}
\Omega= \mathbb{C}_+ \cup \mathbb{C_-} 
\end{equation}
consists of $Y$-regular points of $E$.
\end{rmk}

Now, we aim to construct three different RKHSs which will be denoted by $\mathcal{H}_1$, $\mathcal{H}_2$ and $\mathcal{H}_3$ from a simple, closed, densely defined, symmetric operator $E$ with infinite deficiency indices on the Hilbert space $\mathfrak{X}$. We will observe that all the three spaces are $\mathcal{H}_0(U)$ spaces and they serve as functional models for the operator $E$.

By $\eqref{Equation 3.2}$, for every $f \in \mathfrak{X}$, there exists a unique $g$ (depending on $\xi$) and a projection operator $P_{Y}(\xi)$ such that\begin{equation}
f=(E- \xi I)g + P_{Y}(\xi) f  \hspace{0.3cm} \text{for all} \hspace{0.1cm} \xi \in \Omega.
\end{equation}
The operator $P_{Y}(\xi)$ is a bounded linear operator for all $\xi \in \Omega$. Fixing $f$, we consider a map from $\Omega $ to $\mathfrak{X}\oplus \mathfrak{X}$ defined by \begin{equation}
 \xi \longrightarrow \begin{bmatrix} 
	P_{Y}(\xi)f  \\
	\frac{1}{\pi i} \tau_{Y}(\xi)f \\
	\end{bmatrix},
\end{equation}
where $\tau_{Y}(\xi)\in B(\mathfrak{X})$ is such that $\tau_{Y}(\xi)f=(E-\xi)^{-1}(I-P_{Y} (\xi))f$. We denote these $\mathfrak{X} \oplus \mathfrak{X}$ valued functions by $$\begin{bmatrix} 
	f_Y  \\
	\tilde{f}_Y \\
	\end{bmatrix}
	 $$
which are defined as \begin{equation}
\begin{bmatrix} 
	f_Y  \\
	\tilde{f}_Y \\
	\end{bmatrix}
	(\xi)=\begin{bmatrix} 
	P_{Y}(\xi)f  \\
	\frac{1}{\pi i} \tau_{Y}(\xi)f \\
	\end{bmatrix}.
\end{equation}
Define \(\mathcal{H}_1\) as follows:
\[
\mathcal{H}_1 = \left\lbrace \begin{bmatrix} 
	f_Y  \\
	\tilde{f}_Y \\
	\end{bmatrix} : f \in \mathfrak{X} \right\rbrace.
\]
The space \(\mathcal{H}_1\) consists of vector valued analytic functions defined on \(\Omega\). Since \(E\) is simple, the map \(\Psi: \mathfrak{X} \rightarrow \mathcal{H}_1\) given by
\begin{equation} \label{Equation 3.11}
f \mapsto \frac{1}{\sqrt{2}} \begin{bmatrix} 
	f_Y  \\
	\tilde{f}_Y \\
	\end{bmatrix}
\end{equation}
is injective. Therefore, \(\mathcal{H}_1\) is the vector space of analytic vector valued functions with respect to the standard scalar multiplication and pointwise addition.	
	 Define the inner product in $\mathcal{H}_1$ as follows:
\begin{equation}
\left\langle \begin{bmatrix} 
	f_Y  \\
	\tilde{f}_Y \\
\end{bmatrix}, \begin{bmatrix} 
	g_Y  \\
	\tilde{g}_Y \\
\end{bmatrix} \right\rangle_{\mathcal{H}_1} = 2 \langle f, g \rangle_\mathfrak{X}.
\end{equation}
With respect to this inner product, $\Psi$ defines a unitary operator from $\mathfrak{X}$ to $\mathcal{H}_1$ given by $(\ref{Equation 3.11})$. Hence, $\mathcal{H}_1$ is a Hilbert space. Next, we demonstrate that $\mathcal{H}_1$ is an RKHS by showing that the pointwise evaluations are bounded. Since $P_Y(\xi)$ is a bounded operator on $\mathfrak{X}$, we have
\begin{equation} \label{Equation 3.13}
\Vert P_Y(\xi)f \Vert \leq \Vert P_Y(\xi) \Vert \Vert f \Vert_\mathfrak{X}.
\end{equation}
Furthermore, \begin{eqnarray}
\Vert \frac{1}{\pi i} \tau_Y(\xi)f\Vert &=& \frac{1}{\pi} \Vert (E-\xi)^{-1}(f-P_Y(\xi)f) \Vert\\  
                  & \leq & \frac{1}{\pi} \Vert (E-\xi)^{-1}\Vert \Vert \left(I-  P_Y(\xi)\right)f \Vert  \\ \label{Equation 3.16}
&\leq & \frac{1}{\pi} \Vert (E-\xi)^{-1}\Vert \left( 1+ \Vert P_Y(\xi) \Vert \right) \Vert f \Vert_\mathfrak{X}.
\end{eqnarray}
Utilizing $(\ref{Equation 3.13})$ and $(\ref{Equation 3.16})$, we have \begin{eqnarray}
\biggl\Vert \begin{bmatrix} 
	f_Y  \\
	\tilde{f}_Y\\
	\end{bmatrix} (\xi) \biggl\Vert &=& \biggl\Vert \begin{bmatrix} 
	P_{Y}(\xi)f  \\
	\frac{1}{\pi i} \tau_{Y}(\xi)f \\
	\end{bmatrix} \biggl\Vert \\
	&\leq & c_1 \Vert f \Vert_{\mathfrak{X}} \\ 
	&=& c_2 \biggl\Vert \begin{bmatrix} 
	f_Y  \\
	\tilde{f}_Y \\
	\end{bmatrix}  \biggl\Vert_{\mathcal{H}_1},
\end{eqnarray}
where $c_1$ and $c_2$ are some constants. We define $J_1$ on $\mathfrak{X} \oplus \mathfrak{X}$ as $$
J_1(f,g)=(g,f) \hspace{0.2cm} \text{for all}~ f,g \in \mathfrak{X}.
$$ In matrix operator form, $J_1$ is represented by the following signature operator: $$J_1 =\begin{bmatrix}
0 & I \\ 
I & 0
\end{bmatrix}  . $$
In the forthcoming theorem, which is motivated by \cite[Theorem 8.3]{alpay}, we will establish the relationship between the RKHS as constructed from a simple, closed, densely defined, symmetric operator with infinite deficiency indices and the de Branges space \(\mathcal{H}(U)\).

\begin{thm} \label{Theorem 3.3}
Consider a simple, closed, densely defined, symmetric operator $E$ on the Hilbert space $\mathfrak{X}$ with infinite deficiency indices. Let $\mathcal{H}_1$ denote the RKHS constructed as described above. Then $\mathcal{H}_1$ is a $\mathcal{H}(U)$ space, where $U(\lambda)$ is an analytic operator valued function satisfying $\eqref{Equation 2.2}$ on $\Omega$, with the signature operator $J = J_1$.
\end{thm}
\begin{proof}
In order to demonstrate that $\mathcal{H}_1$ is a $\mathcal{H}(U)$ space, we employ the characterization as provided by Theorem $\ref{Theorem 3.1}$. First, we establish that $\mathcal{H}_1$ is $R_\alpha$ invariant for every $\alpha \in \Omega$. Consider the action of $R_\alpha$ on a function in $\mathcal{H}_1$:
\begin{eqnarray}
\left(R_{\alpha} \begin{bmatrix} 
    f_Y  \\
    \tilde{f}_Y \\
    \end{bmatrix}\right) (\xi)\nonumber
&=&  \begin{bmatrix} 
    (R_{\alpha}f_Y)(\xi)  \\
    (R_{\alpha}\tilde{f}_Y)(\xi) \\ 
    \end{bmatrix}\\ \nonumber
\vspace{0.2cm}       
&=&  \begin{bmatrix} 
    \frac{f_Y(\xi)-f_Y(\alpha)}{\xi-\alpha}  \\
    \frac{\tilde{f}_Y(\xi)- \tilde{f}_Y(\alpha)}  {\xi-\alpha}\\
    \end{bmatrix}\\    
\vspace{0.2cm}       
&=&  \begin{bmatrix} \label{Equation 3.20}
    \frac{P_Y(\xi)f-P_Y(\alpha)f}{\xi-\alpha}  \\
    \frac{\tau_Y(\xi)f- \tau_Y(\alpha)f}{\pi      i(\xi-\alpha)}\\
    \end{bmatrix}.
\end{eqnarray}
Now, let $g$ and $h$ be such that 
$$
f
=(E-\xi I)g+P_Y(\xi)f
=(E-\alpha I)h+P_Y(\alpha)f.
$$
This implies $$ h=\frac{P_Y(\xi)f - P_Y(\alpha)f}{\xi-\alpha}+(E-\xi I)\left( \frac{g-h}{\xi-\alpha}\right).$$
From $\eqref{Equation 3.20}$, we obtain $$  \left(R_{\alpha} \begin{bmatrix} 
    f_Y  \\
    \tilde{f}_Y \\
    \end{bmatrix}\right) (\xi)= \begin{bmatrix} 
    P_Y(\xi)h  \\
    \frac{\tau_Y(\xi)h}{\pi i} \\
    \end{bmatrix}    =  \begin{bmatrix} 
    h_Y  \\
    \tilde{h}_Y \\
    \end{bmatrix} (\xi).$$
Next, we aim to demonstrate the validity of the de Branges identity for all $\alpha, \beta \in \Omega$ and $\begin{bmatrix} 
f_Y  \\
\tilde{f}_Y \\
\end{bmatrix},  \begin{bmatrix} 
g_Y  \\
\tilde{g}_Y \\
\end{bmatrix} \in \mathcal{H}_1$: \begin{multline} \label{Equation 3.21}
\left\langle R_\alpha  \begin{bmatrix} 
	f_Y  \\
	\tilde{f}_Y\\
	\end{bmatrix},  \begin{bmatrix} 
	g_Y  \\
	\tilde{g}_Y \\
	\end{bmatrix} \right\rangle  -\left\langle  \begin{bmatrix} 
	f_Y  \\
	\tilde{f}_Y \\
	\end{bmatrix}, R_{\beta} \begin{bmatrix} 
	g_Y  \\
	\tilde{g}_Y \\
	\end{bmatrix} \right\rangle - (\alpha- \overline{\beta}) \left\langle R_{\alpha}\begin{bmatrix} 
	f_Y  \\
	\tilde{f}_Y \\
	\end{bmatrix},  R_{\beta}\begin{bmatrix} 
	g_Y  \\
	\tilde{g}_Y \\
	\end{bmatrix} \right\rangle 
	\\
	= 2\pi i \begin{bmatrix} 
	g_Y  \\
	\tilde{g}_Y \\
	\end{bmatrix}^*(\beta) J \begin{bmatrix} 
	f_Y  \\
	\tilde{f}_Y\\
	\end{bmatrix}(\alpha).
	\end{multline}
The left-hand side (L.H.S.) of $\eqref{Equation 3.21}$ is expressed as
\begin{multline} \label{Equation 3.22}
\left\langle   \begin{bmatrix} 
	h_Y  \\
	\tilde{h}_Y \\
	\end{bmatrix},  \begin{bmatrix} 
	g_Y  \\
	\tilde{g}_Y \\
	\end{bmatrix} \right\rangle  -\left\langle  \begin{bmatrix} 
	f_Y  \\
	\tilde{f}_Y \\
	\end{bmatrix}, \begin{bmatrix} 
	l_Y  \\
	\tilde{l}_Y \\
	\end{bmatrix} \right\rangle - (\alpha- \overline{\beta}) \left\langle \begin{bmatrix} 
	h_Y  \\
	\tilde{h}_Y \\
	\end{bmatrix},  \begin{bmatrix} 
	l_Y  \\
	\tilde{l}_Y \\
	\end{bmatrix} \right\rangle 
\\
= 2\langle h,g \rangle	-2\langle f,l \rangle-2 (\alpha- \overline{\beta})\langle h, l \rangle	
\end{multline}	
where $h$ and $l$ are defined as 
\begin{equation} \label{Equation 3.23}
f=(E-\alpha I)h +P_Y(\alpha)f \hspace{0.3cm}\text{and} \hspace{0.3cm}
g=(E-\beta I)l+P_Y(\beta)g.
\end{equation}
Observe the following two equalities: $$ 2\langle h,g\rangle =2 \langle h,g-P_Y(\beta)g \rangle+ 2\langle h, P_Y(\beta)g \rangle$$ and \begin{eqnarray*}
\langle h,g-P_Y(\beta)g \rangle &=& \langle h,(E- \beta I)l \rangle \\
&=& \langle Eh,l \rangle- \langle \overline{\beta} h,l\rangle\\
&=& \langle (E-\alpha I)h,l \rangle- \langle (\overline{\beta}-\alpha) h,l\rangle 
\end{eqnarray*} Substituting these into $\eqref{Equation 3.22}$, we obtain $$\text{L.H.S. of $(\ref{Equation 3.21})$}=2\langle h,P_Y(\beta)g \rangle - 2 \langle P_Y(\alpha)f,l \rangle.$$ 
Evaluating R.H.S. of $(\ref{Equation 3.21})$ with $J= \begin{bmatrix}
0 & I \\ 
I & 0
\end{bmatrix}$, we get 
\begin{eqnarray*}
2\pi i  \begin{bmatrix} 
	g_Y  \\
	\tilde{g}_Y\\
	\end{bmatrix}^*(\beta) J  \begin{bmatrix} 
	f_Y  \\
	\tilde{f}_Y \\
	\end{bmatrix}(\alpha) 
&=& 2\pi i  \left\langle J \begin{bmatrix} 
	f_Y  \\
	\tilde{f}_Y \\
	\end{bmatrix}(\alpha) ,  \begin{bmatrix} 
	g_Y  \\
	\tilde{g}_Y \\
	\end{bmatrix}(\beta) \right\rangle \\
&=&	 2 \pi i  \left\langle  \begin{bmatrix}
0 & I \\ 
I & 0
\end{bmatrix} 
\begin{bmatrix} 
	P_Y(\alpha)f  \\
	\frac{\tau_Y(\alpha)f}{\pi i} \\
	\end{bmatrix} ,  \begin{bmatrix} 
	P_Y(\beta)g  \\
	\frac{\tau_Y(\beta)g}{\pi i} \\
	\end{bmatrix} \right\rangle \\
&=& 2\langle \tau_Y(\alpha)f,P_Y(\beta)g \rangle - 2 \langle P_Y(\alpha)f,\tau_Y(\beta)g \rangle\\
&=& 2\langle h,P_Y(\beta)g \rangle - 2 \langle P_Y(\alpha)f,l \rangle.
\end{eqnarray*}
Hence, the result is proved.
\end{proof}
Now, we will observe that the space \( \mathcal{H}_1 \) is a \( \mathcal{H}_0 (U) \) space. Before proceeding, we present a lemma that will be used in the proof.

\begin{lemma}\label{Lemma 3.4}
Let \( I - U(\alpha) \) be compact for \( \alpha = \lambda,~ \omega \), where \( \lambda \neq \bar{\omega} \). Then, the kernel function \( K_{\omega}(\lambda) \), given by $\eqref{Equation 2.3}$, is compact.
\end{lemma}

\begin{proof}
Let \( I - U(\lambda) = L \) and \( I - U(\omega) = L' \), where \( L \) and \( L' \) are compact operators. Then,
\[
K_{\omega}(\lambda) = \frac{J - U(\lambda) J U(\omega)^*}{-2 \pi i (\lambda - \bar{\omega})}.
\]
Substituting \( U(\lambda) = I - L \) and \( U(\omega) = I - L' \), we get:
\[
K_{\omega}(\lambda) = \frac{J - (I - L) J (I - L')^*}{-2 \pi i (\lambda - \bar{\omega})}.
\]
Expanding this expression:
\[
K_{\omega}(\lambda) = \frac{J L'^* + L J - L J L'^*}{-2 \pi i (\lambda - \bar{\omega})}.
\]
Since the sum and product of compact operators are compact, it follows that \( K_{\omega}(\lambda) \) is compact.
\end{proof}


\begin{rmk} \label{Remark 3.5}
A straightforward argument based on the definition of compactness shows that if \( A: \mathfrak{X} \rightarrow \mathfrak{X} \oplus \mathfrak{X} \) is defined by
\[
Af = \begin{bmatrix}
A_1 f \\
A_2 f
\end{bmatrix}
\]
and $A$ is compact, then both \( A_1 \) and \( A_2 \) must be compact operators.
\end{rmk}


\begin{thm}\label{Theorem 3.6}
Under the hypothesis of Theorem \ref{Theorem 3.3}, the space \( \mathcal{H}_1 \) is a \( \mathcal{H}_0 (U) \) space.
\end{thm}
\begin{proof}
It only needs to be checked that \( I - U(\lambda) \) is not compact for all \( \lambda \in \Omega \). Here, \( \delta_{\lambda}: \mathcal{H}_1 \rightarrow \mathfrak{X} \oplus \mathfrak{X} \) is defined by
\[
\delta_{\lambda}\begin{bmatrix} 
f_Y  \\
\tilde{f}_Y
\end{bmatrix}
= \begin{bmatrix} 
f_Y  \\
\tilde{f}_Y
\end{bmatrix}(\lambda)
= \begin{bmatrix} 
P_{Y}(\lambda) f  \\
\frac{1}{\pi i} \tau_{Y}(\lambda) f
\end{bmatrix}.
\]
Since \( Y \) is infinite-dimensional, this implies that the operator \( P_{Y}(\lambda) \) is not compact. Hence, by previous remark, \( \delta_{\lambda} \) is not compact. Now, since \( K_{\lambda}(\lambda) = \delta_{\lambda} \delta_{\lambda}^* \), \( K_{\lambda}(\lambda) \) is compact if and only if \( \delta_{\lambda} \) is compact. By Lemma \ref{Lemma 3.4}, it follows that \( I - U(\lambda) \) is not compact for all \( \lambda \in \Omega \).
\end{proof}

Corresponding to the space $\mathcal{H}_1$, we define the space $\mathcal{N}_1$ as 
\begin{equation}\label{Equation 3.24}
\mathcal{N}_1= \{f_Y : f \in \mathfrak{X} \}
\end{equation}
where \(f_Y(\xi) = P_Y(\xi)f$ for all $\xi \in \Omega\), with respect to the following inner product
\begin{equation} \label{Equation 3.25}
\langle f_Y, g_Y \rangle_{\mathcal{N}_1} = \langle f, g \rangle_{\mathfrak{X}}.
\end{equation}
Define the operator $\mathfrak{D}_1$ by \begin{eqnarray} \label{Equation 3.26}
\mathfrak{D}_1\begin{bmatrix} 
                       	f_Y  \\
                    	Z_1f_Y \\
                    	\end{bmatrix}   = \begin{bmatrix} 
                       	                  Tf_Y  \\
                    	                  Z_1Tf_Y \\
                    	                   \end{bmatrix} ~~~~\mbox{for all}~~ f \in \mathfrak{X} 
\end{eqnarray} where the operator $Z_1$ defined on the space $\mathcal{N}_1$ is given by 
 \begin{equation}
 \langle (Z_1f_Y)(\xi), l \rangle_{\mathfrak{X}}= \frac{1}{\pi i}\langle R_{\xi}f_Y, l_Y \rangle _{\mathcal{N}_1}
\hspace{0.3cm} \text{for all}~ f_Y \in \mathcal{N}_1, \xi \in \Omega, l \in \mathfrak{X},
\end{equation} 
and the operator $T$ is the multiplication operator by the independent variable defined on the space $\mathcal{N}_1$.
We will come back to the role of the operator $\mathfrak{D}_1$ in the Theorem \ref{Theorem 3.12} in the context of functional models for the operator $E$.

Next, we will construct another space $\mathcal{H}_2$ from a simple, closed, densely defined, symmetric operator $E$ with infinite deficiency indices which consists of $Y \oplus Y$ valued functions defined on $\Omega$. We will observe that this space is also an \( \mathcal{H}_0 (U) \) space. Consider the mapping from $\Omega$ to $Y \oplus Y$ defined as follows:
\begin{equation}
\xi \longrightarrow \begin{bmatrix}
P_{Y}(\xi)f \\
\frac{1}{\pi i} \Pi_{Y} \bigl(\tau_{Y}(\xi)f \bigl) \
\end{bmatrix}
\end{equation}
Here, $\Pi_{Y}(g)$ denotes the orthogonal projection of $g \in \mathfrak{X}$ onto the closed, infinite dimensional subspace $Y$.
Define $\mathcal{H}_2$ as follows: $$\mathcal{H}_2= \biggl \lbrace \begin{bmatrix} 
	f_Y  \\
	\tilde{f}_Y \\
	\end{bmatrix} : f \in \mathfrak{X}\biggl \rbrace$$ such that  $$\begin{bmatrix} 
	f_Y  \\
	\tilde{f}_Y \\
	\end{bmatrix}
	(\xi)=\begin{bmatrix} 
	P_{Y}(\xi)f  \\
	\frac{1}{\pi i} \Pi_{Y} \bigl(\tau_{Y}(\xi)f \bigl) \\
	\end{bmatrix}.$$
Similar to $\mathcal{H}_1$, we observe that $\mathcal{H}_2$ is a Hilbert space of vector valued analytic functions with inner product defined as \begin{equation}
	\left\langle   \begin{bmatrix} 
	f_Y  \\
	\tilde{f}_Y \\
	\end{bmatrix},   \begin{bmatrix} 
	g_Y  \\
	\tilde{g}_Y \\
	\end{bmatrix} \right\rangle_{\mathcal{H}_2} =2 \langle f,g \rangle_\mathfrak{X}.
	\end{equation}
It can be similarly verified that $\mathcal{H}_2$ is an RKHS. The following theorem is analogous to Theorem $\ref{Theorem 3.3}$.

\begin{thm}\label{Theorem 3.7}
Consider a simple, closed, densely defined, symmetric operator $E$ on the Hilbert space $\mathfrak{X}$ with infinite deficiency indices. Let $\mathcal{H}_2$ denote the RKHS constructed as described above. Then $\mathcal{H}_2$ is a $\mathcal{H}_0(U)$ space, with the signature operator $J = J_1$.
\end{thm}
\begin{proof}
The verification of this theorem follows the same approach as in Theorem $\ref{Theorem 3.3}$. We begin by showing that $\mathcal{H}_2$ is $R_\alpha$ invariant for every $\alpha \in \Omega$.

\begin{eqnarray*}
\left(R_{\alpha} \begin{bmatrix} 
	f_Y  \\
	\tilde{f}_Y \\
	\end{bmatrix}\right) (\xi)\nonumber	
&=&  \begin{bmatrix}
	\frac{P_Y(\xi)f-P_Y(\alpha)f}{\xi-\alpha}  \\
	\frac{\Pi_{Y} \bigl(\tau_Y(\xi)f \bigl)- \Pi_{Y} \bigl(\tau_Y(\alpha)f \bigl)}{\pi i(\xi-\alpha)}\\
	\end{bmatrix}\\
&=&	\begin{bmatrix} 
	P_Y(\xi)h  \\
	\frac{\Pi_{Y} \bigl(\tau_Y(\xi)h \big)}{\pi i} \\
	\end{bmatrix} 
=\begin{bmatrix} 
	h_Y  \\
	\tilde{h}_Y\\
	\end{bmatrix} (\xi)
\end{eqnarray*}
where$$
f
=(E-\xi I)g+P_Y(\xi)f
=(E-\alpha I)h+P_Y(\alpha)f$$
and $$ h=\frac{P_Y(\xi)f - P_Y(\alpha)f}{\xi-\alpha}+(E-\xi I)\left( \frac{g-h}{\xi-\alpha}\right).$$
Next, we show that ($\ref{Equation 3.21}$) holds. L.H.S. of ($\ref{Equation 3.21}$) is given by: 
\begin{eqnarray*}
     2\langle h,P_Y(\beta)g \rangle &-& 2 \langle P_Y(\alpha)f,l \rangle \\
&=&   2\langle h + \Pi_{Y}(h) -\Pi_{Y}(h),P_Y(\beta)g \rangle - 2 \langle P_Y(\alpha)f,l + \Pi_{Y}(l) -\Pi_{Y}(l) \rangle \\
&=&      2\langle \Pi_{Y}(h) , P_Y(\beta)g \rangle - 2 \langle P_Y(\alpha)f,\Pi_{Y}(l) \rangle,
\end{eqnarray*}
where $l$ is defined in ($\ref{Equation 3.23}$). R.H.S. of ($\ref{Equation 3.21}$) is evaluated as: \begin{eqnarray*}
2\pi i  \begin{bmatrix} 
	g_Y  \\
	\tilde{g}_Y \\
	\end{bmatrix}^*(\beta) J  \begin{bmatrix} 
	f_Y  \\
	\tilde{f}_Y \\
	\end{bmatrix}(\alpha) 
&=& 2\pi i  \left\langle J \begin{bmatrix} 
	f_Y  \\
	\tilde{f}_Y \\
	\end{bmatrix}(\alpha) ,  \begin{bmatrix} 
	g_Y  \\
	\tilde{g}_Y \\
	\end{bmatrix}(\beta) \right\rangle \\
&=&	 2 \pi i  \left\langle  \begin{bmatrix}
0 & I \\ 
I & 0
\end{bmatrix} 
\begin{bmatrix} 
	P_Y(\alpha)f  \\
	\frac{\Pi_{Y} \bigl(\tau_Y(\alpha)f \bigl)}{\pi i} \\
	\end{bmatrix} ,  \begin{bmatrix} 
	P_Y(\beta)g  \\
	\frac{\Pi_{Y}\bigl(\tau_Y(\beta)g\bigl)}{\pi i} \\
	\end{bmatrix} \right\rangle \\
&=& 2\langle \Pi_{Y}\bigl(\tau_Y(\alpha)f \bigl) ,P_Y(\beta)g \rangle - 2 \langle P_Y(\alpha)f,\Pi_{Y}\bigl(\tau_Y(\beta)g \bigl) \rangle\\
&=& 2\langle \Pi_{Y}(h) ,P_Y(\beta)g \rangle - 2 \langle P_Y(\alpha)f,\Pi_{Y}(l) \rangle.
\end{eqnarray*}
This completes the verification of $R_\alpha$ invariance of $\mathcal{H}_2$ and the de Branges identity. Also, the fact that $I-U$ is not compact follows similarly from Theorem \ref{Theorem 3.6}.
\end{proof}

Similar to the construction of the space $\mathcal{N}_1$ and the operator $\mathfrak{D}_1$, we define the space $\mathcal{N}_2$ corresponding to the space $\mathcal{H}_2$ as 
\begin{equation} \label{Equation 3.30}
\mathcal{N}_2= \{f_Y : f \in \mathfrak{X} \}
\end{equation}
where \(f_Y(\xi) = P_Y(\xi)f$ for all $\xi \in \Omega\), with respect to the following inner product
\begin{equation} 
\langle f_Y, g_Y \rangle_{\mathcal{N}_2} = \langle f, g \rangle_{\mathfrak{X}},
\end{equation}
and define the operator $\mathfrak{D}_2$ by \begin{eqnarray} \label{Equation 3.32}
\mathfrak{D}_2\begin{bmatrix} 
                       	f_Y  \\
                    	Z_2f_Y \\
                    	\end{bmatrix}   = \begin{bmatrix} 
                       	                  Tf_Y  \\
                    	                  Z_2Tf_Y \\
                    	                   \end{bmatrix} ~~~~\mbox{for all}~~ f \in \mathfrak{X} 
\end{eqnarray} where the operator $Z_2$ defined on the space $\mathcal{N}_2$ is given by 
 \begin{equation}
 \langle (Z_2f_Y)(\xi), l \rangle_{\mathfrak{X}}= \frac{1}{\pi i}\langle R_{\xi}f_Y, (\Pi_Yl)_Y \rangle_{\mathcal{N}_2}
\hspace{0.3cm} \text{for all}~ f_Y \in \mathcal{N}_2, \xi \in \Omega, l \in \mathfrak{X},
\end{equation} 
and the operator $T$ is the multiplication operator by the independent variable defined on the space $\mathcal{N}_2$.
We will come back to the role of the operator $\mathfrak{D}_2$ in the Theorem \ref{Theorem 3.12} in the context of functional models for the operator $E$.

In the following part of this section, which is motivated from \cite{dd} and \cite{mahapatra}, we will establish that the operator \(E\) on \(\mathfrak{X}\) is unitarily equivalent to the operator \(\mathfrak{D}\) (as will be defined by ($\ref{Equation 3.39}$)) on a vector valued de Branges space \(\mathcal{H}_3\).\\
Let $G : \mathfrak{X} \rightarrow \mathfrak{X}$ be a one-one bounded linear operator on $\mathfrak{X}$ with $\mathrm{rng}(G)=Y$. Define the operators $\mathfrak{P}$ and $\mathfrak{R}$ on $\mathfrak{X}$ such that \begin{equation}
\mathfrak{P}f(\xi)=G^{-1}P_{Y}(\xi)f~ ~\mbox{and}~~\mathfrak{R}f(\xi)=\frac{1}{\pi i} G^{*} \tau_{Y}(\xi)f.
\end{equation}
The space $\mathcal{H}_3$ is defined as the  space of vector valued functions on $\Omega$ as follows: $$\mathcal{H}_3= \biggl \lbrace \begin{bmatrix} 
	f_Y  \\
	\tilde{f}_Y \\
	\end{bmatrix} : f \in \mathfrak{X}\biggl \rbrace$$ where  $$\begin{bmatrix} 
	f_Y  \\
	\tilde{f}_Y \\
	\end{bmatrix}
	(\xi)=\begin{bmatrix} 
	\mathfrak{P}f(\xi)  \\
	\mathfrak{R}f(\xi) \\
	\end{bmatrix}
	=
	         \begin{bmatrix} 
	G^{-1}P_{Y}(\xi)f  \\
	\frac{1}{\pi i} G^{*} \tau_{Y}(\xi)f \\
	\end{bmatrix}.$$
The space is equipped with the standard pointwise addition and scalar multiplication. Analogous to $\mathcal{H}_1$ and $\mathcal{H}_2$ , it can be shown that $\mathcal{H}_3$ is an RKHS and similar to Theorem $\ref{Theorem 3.3}$, it can be observed that $\mathcal{H}_3$ is a $\mathcal{H}(U)$ 
space with the signature operator $J_1$.\\
Let 
\[
\mathcal{N} = \{ f_Y : f \in \mathfrak{X} \},
\]
where \(f_Y(\xi) = G^{-1}P_Y(\xi)f$ for all $\xi \in \Omega\), with respect to the following inner product
\begin{equation} \label{Equation 3.35}
\langle f_Y, g_Y \rangle = \langle f, g \rangle.
\end{equation}
\begin{thm}\label{Theorem 3.8} The space $\mathcal{N}$ is a de Branges space $\mathcal{B}(\mathfrak{E})$ of vector valued functions that are analytic on $\Omega$, provided there exists at least one \(\alpha \in  \Omega \cap \mathbb{C}_+\) such that the following conditions hold true:
\begin{equation} \label{Equation 3.36}
(1)~\dim(\mathfrak{M}_{\alpha} \cap \mathfrak{M}_{\omega}^{\perp}) < \infty~ \mbox{and}~ \mathfrak{M}_{\alpha} + \mathfrak{M}_{\omega}^{\perp} ~\mbox{is closed for all}~ \omega \in  \Omega \cap \mathbb{C}_-,
 \end{equation}
 \begin{equation} \label{Equation 3.37}
(2)~\dim(\mathfrak{M}_{\bar{\alpha}} \cap \mathfrak{M}_{\omega}^{\perp}) < \infty ~ \mbox{and}~ \mathfrak{M}_{\bar{\alpha}} + \mathfrak{M}_{\omega}^{\perp} ~\mbox{is closed for all}~ \omega \in \Omega \cap \mathbb{C}_+.
\end{equation}
Additionally, the operator \(E\) is unitarily equivalent to the multiplication operator \(T\) on \(\mathcal{N}\), where \(Tf(\lambda) = \lambda f(\lambda)$ for all $f \in \mathcal{N}\).
\end{thm}
\begin{proof}We have included this proof in the Appendix.
\end{proof}
\begin{rmk} Here, we want to remark that the proof of the above theorem: $\mathcal{N}$ is a $\mathcal{B}(\mathfrak{E})$ space and $E$ is unitarily equivalent to the multiplication operator by the independent variable, essentially follows from Theorem $10.3$ in \cite{mahapatra}, where the required hypothesis $\mathfrak{M}_{\alpha} + \mathfrak{M}_{\omega}^{\perp}$ is closed for all $ \omega \in  \Omega \cap \mathbb{C}_-$ and $\mathfrak{M}_{\bar{\alpha}} + \mathfrak{M}_{\omega}^{\perp} $ is closed for all $\omega \in \Omega \cap \mathbb{C}_+$, were absent. We have stated this theorem by adding these slight modifications of the hypothesis. The proof is essentially the same and for convenience we have also added the proof in the Appendix.
   \end{rmk}
\begin{rmk}Note that the pairs $(\mathfrak{M}_{\alpha} , \mathfrak{M}_{\omega}^{\perp})$ and $(\mathfrak{M}_{\bar{\alpha}} , \mathfrak{M}_{\omega}^{\perp})$ satisfying the conditions \eqref{Equation 3.36} and \eqref{Equation 3.37} are semi-Fredholm. Recall that a pair $(M,N)$, where $M$ and $N$ are closed subspaces in a Banach space, is called semi-Fredholm if $M+N$ is a closed subspace and atleast one of the $\dim(M \cap N)$ and $\mathrm{codim}(M+N)$ is finite. For a detailed discussion on pairs of semi-Fredholm subspaces, we refer to \cite[Chapter $4$, Section $4$]{kato}.
\end{rmk}
Consider the following equalities:
\begin{eqnarray*}
(\Psi(Ef)(\xi)) &=&  \frac{1}{\sqrt{2}}
                    \begin{bmatrix} 
                       	(Ef)_Y(\xi)  \\
                    	\tilde{(Ef)}_Y(\xi) \\
                    	\end{bmatrix}
               =  \frac{1}{\sqrt{2}} \begin{bmatrix} 
	                   G^{-1}P_{Y}(\xi)(Ef)  \\   	                      \frac{1}{\pi i} G^*\tau_{Y}(\xi)(Ef) \\
                	\end{bmatrix}	\\
                	 &=& \frac{1}{\sqrt{2}} \begin{bmatrix} 
	             \mathfrak{P}Ef(\xi)  \\
                    	\mathfrak{R}Ef(\xi) \\
	                  \end{bmatrix}
               \overset{\mbox{\ding{172}}}{=} \frac{1}{\sqrt{2}}\begin{bmatrix} 
                 	( \mathfrak{P}(Ef))(\xi)  \\
                 	(Z(\mathfrak{P}(Ef)))(\xi) \\
                 	\end{bmatrix}	\\ 
              & \overset{\mbox{\ding{173}}}{=}&  	\frac{1}{\sqrt{2}}\begin{bmatrix} 
                 	T \mathfrak{P}f(\xi)  \\
                 	ZT\mathfrak{P}f(\xi) \\
                 	\end{bmatrix}	,     	             
\end{eqnarray*}
where $Z$ is an $\mathfrak{X}$-valued function defined on the space $\mathcal{B}(\mathfrak{E})$ and given by 
 \begin{equation} \label{Equation 3.38}
 \langle (Zf_Y)(\xi), l \rangle_{\mathfrak{X}}= \frac{1}{\pi i}\langle R_{\xi}f_Y, (Gl)_Y \rangle_{\mathcal{B}(\mathfrak{E})} 
\hspace{0.3cm} \text{for all}~ f_Y \in \mathcal{B}(\mathfrak{E}), \xi \in \Omega, l \in \mathfrak{X}.
\end{equation}
The equality \ding{172} holds true as follows:
\begin{eqnarray*}
\langle (Z \mathfrak{P}f)(\xi), l \rangle
    &=& \frac{1}{\pi i} \langle R_{\xi} \mathfrak{P}f, (Gl)_Y \rangle  \\
    &=& \frac{1}{\pi i} \langle \mathfrak{P}g, \mathfrak{P}Gl \rangle_{\mathcal{B}(\mathfrak{E})} \\
    &=& \frac{1}{\pi i} \langle g, Gl \rangle_{\mathfrak{X}} \\
    &=& \frac{1}{\pi i} \langle G^*g, l \rangle \\
    &=& \langle \mathfrak{R}f(\xi), l \rangle \quad \text{for all}~ f \in \mathfrak{X}.
\end{eqnarray*}
The reason for equality \ding{173} is stated as follows: The operator $\mathfrak{P}$ on $\mathfrak{X}$ is a unitary operator since $E$ is simple, and for all $f \in \mathfrak{X}$ and $\xi \in \Omega$, the equality $\mathfrak{P}Ef(\xi)= T\mathfrak{P}f(\xi)$ holds true since
\begin{eqnarray*}
\mathfrak{P}Ef(\xi) &=&  G^{-1}P_{Y}(\xi)(Ef)\\
&=& G^{-1}P_{Y}(\xi)(Ef- \xi f+ \xi f)\\
&=& G^{-1}P_{Y}(\xi)\biggl(Ef- \xi f+ \xi \bigl( (E-\xi)h+P_{Y}(\xi)f \bigl) \biggl)\\
&=& G^{-1}P_{Y}(\xi)\bigl( (E- \xi)( f+ \xi h)+\xi P_{Y}(\xi)f \bigl)\\
&=& G^{-1}\xi P_{Y}(\xi)f\\
&=& \xi \mathfrak{P}f(\xi)\\
&=& T\mathfrak{P}f(\xi).
\end{eqnarray*}
This implies $$  \Psi Ef(\xi) =\frac{1}{\sqrt{2}} \mathfrak{D}                                  \begin{bmatrix} 
                       	\mathfrak{P}f(\xi)  \\
                    	Z\mathfrak{P}f(\xi) \\
                    	\end{bmatrix} =\frac{1}{\sqrt{2}}\mathfrak{D}                                  \begin{bmatrix} 
	                   G^{-1}P_{Y}(\xi)f  \\   	                      \frac{1}{\pi i} G^*\tau_{Y}(\xi)f \\
                	\end{bmatrix}	=  \mathfrak{D}                                 (\Psi f)(\xi),$$
where \begin{eqnarray} \label{Equation 3.39}
\mathfrak{D}\begin{bmatrix}
                       	f_Y  \\
                    	Zf_Y \\
                    	\end{bmatrix}   = \begin{bmatrix} 
                       	                  Tf_Y  \\
                    	                  ZTf_Y \\
                    	                   \end{bmatrix} ~~~~\mbox{for all}~~ f \in \mathfrak{X} .
\end{eqnarray}            	                   
We summarize the discussions above in the following theorem.
\begin{thm}\label{Theorem 3.11}
Let $E$ be a simple, closed, densely defined, symmetric operator on a Hilbert space $\mathfrak{X}$ with infinite deficiency indices. Let $\mathcal{H}_3$ be the RKHS constructed as described above. Then, the following statements hold:
\begin{itemize}
\item[(1)] $\mathcal{H}_3$ is a $ \mathcal{H}_0(U)$ space, with the signature operator $J= J_1$.
\item[(2)] If there exists at least one \(\alpha \in \Omega \cap \mathbb{C}_+\) such that the following conditions hold true:
\begin{itemize}
\item[i)] \(\dim(\mathfrak{M}_{\alpha} \cap \mathfrak{M}_{\omega}^{\perp}) < \infty\) and $\mathfrak{M}_{\alpha} + \mathfrak{M}_{\omega}^{\perp}$ is closed for all \(\omega \in  \Omega \cap \mathbb{C}_-\),
\item[ii)] \(\dim(\mathfrak{M}_{\bar{\alpha}} \cap \mathfrak{M}_{\omega}^{\perp}) < \infty\) and $\mathfrak{M}_{\bar{\alpha}} + \mathfrak{M}_{\omega}^{\perp}$ is closed for all \(\omega \in \Omega \cap \mathbb{C}_+\),
\end{itemize} then the operator $E$ on $\mathfrak{X}$ is unitarily equivalent to the operator $\mathfrak{D}$ on the de Branges space $\mathcal{H}_3\).
\end{itemize}
\end{thm}
\begin{proof}
The fact that \( \mathcal{H}_3 \) is a \( \mathcal{H}_{0}(U) \) space follows from the observation that since \( Y \) is infinite-dimensional, the operator \( G^{-1}P_{Y}(\xi) \) is not compact. The remainder of the proof follows from the above discussions.
\end{proof}

In the next theorem, we provide functional models for the operator $E$ based on the vector valued de Brange spaces $\mathcal{H}_1$ and $\mathcal{H}_2$.

\begin{thm}\label{Theorem 3.12}Let $E$ be a simple, closed, densely defined, symmetric operator on a Hilbert space $\mathfrak{X}$ with infinite deficiency indices. Let $\mathcal{H}_1$ and $\mathcal{H}_2$ be the RKHS constructed as described above. If there exists at least one \(\alpha \in \Omega \cap \mathbb{C}_+\) such that the following conditions hold true:
\begin{itemize}
\item[i)] \(\dim(\mathfrak{M}_{\alpha} \cap \mathfrak{M}_{\omega}^{\perp}) < \infty\) and $\mathfrak{M}_{\alpha} + \mathfrak{M}_{\omega}^{\perp}$ is closed for all \(\omega \in  \Omega \cap \mathbb{C}_-\),
\item[ii)] \(\dim(\mathfrak{M}_{\bar{\alpha}} \cap \mathfrak{M}_{\omega}^{\perp}) < \infty\) and $\mathfrak{M}_{\bar{\alpha}} + \mathfrak{M}_{\omega}^{\perp}$ is closed for all \(\omega \in \Omega \cap \mathbb{C}_+\),
\end{itemize} then \begin{itemize}
\item[(1)] the operator $E$ on $\mathfrak{X}$ is unitarily equivalent to the operator $\mathfrak{D}_1$, as given by \eqref{Equation 3.26}, on the de Branges space $\mathcal{H}_1\),
\item[(2)] the operator $E$ on $\mathfrak{X}$ is unitarily equivalent to the operator $\mathfrak{D}_2$, as given by \eqref{Equation 3.32}, on the de Branges space $\mathcal{H}_2\).
\end{itemize}
\end{thm}
\begin{proof}
By the arguments similar to those in Theorem \ref{Theorem 3.8}, we deduce that the spaces \( \mathcal{N}_1 \) and \( \mathcal{N}_2 \), as described in \eqref{Equation 3.24} and \eqref{Equation 3.30}, are the de Branges spaces \( \mathcal{B}(\mathfrak{E}) \). For the proof of $(1)$, observe that the following holds true for all \( f_Y \in \mathcal{B}(\mathfrak{E}) \), \( \xi \in \Omega \), and \( l \in \mathfrak{X} \): 
\[
\langle (Z_1 f_Y)(\xi), l \rangle = \frac{1}{\pi i}\langle R_{\xi} f_Y, l_Y \rangle = \frac{1}{\pi i}\langle g_Y, l_Y \rangle =\frac{1}{\pi i} \langle g, l \rangle = \langle \tilde{f}_Y(\xi), l \rangle_{\mathfrak{X}},  
\]  
where $g$ is such that $f = (E - \xi I)g + P_Y(\xi)f$. 

Additionally,
\begin{eqnarray*}
\Psi Ef(\xi) = \frac{1}{\sqrt{2}}
                    \begin{bmatrix} 
                       	(Ef)_Y(\xi)  \\
                    	\tilde{(Ef)}_Y(\xi) \\
                    	\end{bmatrix}
               &=&  \frac{1}{\sqrt{2}} \begin{bmatrix} 
	                   P_{Y}(\xi)(Ef)  \\   	                      \frac{1}{\pi i} \tau_{Y}(\xi)(Ef) \\
                	\end{bmatrix}	\\
                	&=&  \frac{1}{\sqrt{2}} \begin{bmatrix} 
	                   (Tf_{Y})(\xi) \\   	                                                (Z_1Tf_{Y})(\xi) \\
                	\end{bmatrix}	\\
                	&=&\frac{1}{\sqrt{2}} \mathfrak{D}_1                                  \begin{bmatrix} 
                       	f_Y(\xi)  \\
                    	Z_1f_Y(\xi) \\
                    	\end{bmatrix} \\                   
                	&=& \frac{1}{\sqrt{2}}\mathfrak{D}_1 
\begin{bmatrix}
	                 f_{Y}(\xi)  \\   	                                                \tilde{f}_{Y}(\xi) \\
                	\end{bmatrix}=  \mathfrak{D}_1                                 (\Psi f)(\xi).	                	
\end{eqnarray*}
Following the same reasoning as in Theorem \ref{Theorem 3.11}, we conclude that the operator \( E \) on \( \mathfrak{X} \) is unitarily equivalent to the operator \( \mathfrak{D}_1 \) on the de Branges space \( \mathcal{H}_1 \). For the proof of $(2)$, observe that the following holds true for all \( f_Y \in \mathcal{B}(\mathfrak{E}) \), \( \xi \in \Omega \), and \( l \in \mathfrak{X} \): 
\begin{eqnarray*}
\langle (Z_2 f_Y)(\xi), l \rangle = \frac{1}{\pi i}\langle R_{\xi} f_Y, (\Pi_Yl)_Y \rangle 
&=&\frac{1}{\pi i}\langle g_Y, (\Pi_{Y}l)_Y \rangle\\
&=&\frac{1}{\pi i} \langle g, \Pi_{Y}l \rangle\\
&=& \frac{1}{\pi i}\langle \Pi_{Y}g, l \rangle\\
 &=& \langle \tilde{f}_Y(\xi), l \rangle_{\mathfrak{X}},  
\end{eqnarray*} 
where $g$ is such that $f = (E - \xi I)g + P_Y(\xi)f$. The remainder of the proof for $(2)$ follows similarly as in $(1)$.
\end{proof}

We end this section by making a small discussion about another functional model for entire operators which is realized in a de Branges-Rovnyak space.

\begin{thm}
Let $E$ be an entire operator on a Hilbert space $\mathfrak{X}$ with infinite deficiency indices. If there exists at least one \(\alpha \in \mathbb{C}_+\) such that the following conditions hold true:
\begin{itemize}
\item[i)] \(\dim(\mathfrak{M}_{\alpha} \cap \mathfrak{M}_{\omega}^{\perp}) < \infty\) and $\mathfrak{M}_{\alpha} + \mathfrak{M}_{\omega}^{\perp}$ is closed for all \(\omega \in   \mathbb{C}_-\),
\item[ii)] \(\dim(\mathfrak{M}_{\bar{\alpha}} \cap \mathfrak{M}_{\omega}^{\perp}) < \infty\) and $\mathfrak{M}_{\bar{\alpha}} + \mathfrak{M}_{\omega}^{\perp}$ is closed for all \(\omega \in  \mathbb{C}_+\),
\end{itemize}then $E$ is unitarily equivalent to the multiplication operator by the independent variable $\tilde{T}$ on the de Branges-Rovnyak space $\mathcal{H}(S)$, where $S=E_+^{-1}E_- \in S_{B(\mathfrak{X})}(\mathbb{C}_+) $, corresponding to a de Branges operator $(E_- ,E_+)$.
\end{thm}

\begin{proof}
By Theorem \ref{Theorem 3.8}, the operator $E$ is unitarily equivalent to the multiplication operator by the independent variable $T$ on the space $\mathcal{N}$, which is a de Branges space $\mathcal{B}(\mathfrak{E})$. The proof then follows from Theorem $3.10$ in \cite{multi}, which asserts that the operator $$M_{E_+^{-1}}: \mathcal{B}(\mathfrak{E}) \rightarrow \mathcal{H}(S),  \quad M_{E_+^{-1}}(f) = E_+^{-1}f,$$ is a unitary map, and by observing the following holds for all $z \in \mathbb{C}$ and for all $f \in \mathcal{B}(\mathfrak{E})$ :
\begin{eqnarray*}
 ((M_{E_+^{-1}}T)f)(z) &=&(E_+^{-1}Tf)(z) \\
                &=& E_+^{-1}(z)(Tf)(z)\\
                &=& zE_+^{-1}(z)f(z)\\
                &=&(\tilde{T}(E_+^{-1}f))(z)\\
                &=&((\tilde{T}M_{E_+^{-1}})f)(z).
\end{eqnarray*}
Thus, \( M_{E_+^{-1}} T = \tilde{T} M_{E_+^{-1}} \), which establishes that \( E \) is unitarily equivalent to \( \tilde{T} \) on \( \mathcal{H}(S) \).
\end{proof}

\section{Appendix}
In this appendix, we outline the proof of Theorem \ref{Theorem 3.8}. For that, we require the following characterization theorem of $\mathcal{B}(\mathfrak{E})$ spaces.
\begin{thm}\label{Theorem 4.1}
Let $\mathcal{H}$ be an RKHS of $\mathfrak{X}$ valued holomorphic functions that are defined on an open set $\Omega \subseteq \mathbb{C}$ which is symmetric about the real axis with $B(\mathfrak{X})$ valued RK $K_{\omega}(z)$ on $\Omega \times \Omega$. Suppose that there exists a point $\beta \in \Omega \cap \mathbb{C}_+$ such that,
 $$K_{\beta}(z), ~ K_{\bar{\beta}}(z) ~\mbox{are Fredholm operators for all}~ z \in \Omega,$$
 and $$K_{\beta}(\beta), ~ K_{\bar{\beta}}(\bar{\beta}) ~\mbox{are invertible }.$$
 Let $\mathcal{H}_{\omega}= \{ f \in \mathcal{H}: f(\omega)=0 \}$ for each point $\omega \in \Omega$. Then the RKHS is same as a de Branges space $\mathcal{B}(\mathfrak{E})$, based on a de Branges operator $\mathfrak{E}(z)= (E_-(z), E_+(z))$ with $$K_{\omega}(z)= \frac{E_+(z)E_+(\omega)^*-E_-(z)E_-(\omega)^*}{\rho_\omega(z)}  \quad{for}~ z, \omega \in \Omega, z \neq \bar{\omega},$$
 if and only if 
\begin{itemize}
\item[(1)] $R_{\beta}\mathcal{H}_{\beta} \subseteq \mathcal{H}, R_{\bar{\beta}}\mathcal{H}_{\bar{\beta}} \subseteq \mathcal{H}$
\item[(2)] The linear transformation $$T_{\beta}= I+(\beta- \bar{\beta})R_{\beta}: \mathcal{H}_{\beta} \rightarrow \mathcal{H}_{\bar{\beta}}$$ is an isometric isomorphism.
\end{itemize}
\end{thm}
\begin{proof}
This theorem extends a characterization given by de Branges for spaces of scalar valued functions (see \cite{brange}) to vector valued functions based on matrix valued RK and operator valued RK that are presented in \cite{sarkar} and \cite{mahapatra2} respectively. It follows by observing that the proof of Theorem $3.1$ in \cite{mahapatra2} which is given for spaces of entire vector valued functions can be easily adapted to the spaces of vector valued functions that are holomorphic in an open set $\Omega \subseteq \mathbb{C}$ that is symmetric about the real axis.
\end{proof}

\begin{proof}[Proof of Theorem \ref{Theorem 3.8}]
In order to prove the theorem, we will use Theorem \ref{Theorem 4.1} of the characterization of de Branges spaces of vector valued functions that are analytic on $\Omega$. First observe that the space $\mathcal{N}$ consists of vector valued analytic functions defined on $\Omega$. Since $E$ is simple, the map $\psi : \mathfrak{X} \rightarrow \mathcal{N}$ given by $$\psi(f)=f_Y$$ is injective. Therefore, $\mathcal{N}$ is a vector space of analytic vector valued functions with respect to the standard scalar multiplication and pointwise addition. With respect to the inner product \eqref{Equation 3.35}, $\psi$ is a unitary map. Hence, $\mathcal{N}$ is a Hilbert space. Also, the pointwise evaluations are bounded which imply that $\mathcal{N}$ is an RKHS of $\mathfrak{X}$ valued holomorphic functions that are defined on $\Omega$. The rest of the proof is divided into steps. The step $1$ justifies that $K_{\alpha}(z)$ and $K_{\bar{\alpha}}(z)$ are Fredholm operators for all $z \in \Omega$. In step $2$, we will show that $K_{\alpha}(\alpha)$ and $K_{\bar{\alpha}}(\bar{\alpha})$ are invertible operators. Step $3$ serves to justify conditions $(1)$ and $(2)$ of Theorem \ref{Theorem 4.1}.
\\
$\textbf{1.}$ Since $\delta_z(f_Y)= f_Y(z)=G^{-1}P_Y(z)f$,
we have \begin{equation} \label{Equation 4.1}\mathrm{rng}(\delta_z)=\mathfrak{X}~ \mbox{and}~ \ker(\delta_z)=\{ f_Y : f \in \mathrm{rng}(E-zI) \}.
\end{equation} 
This implies \begin{equation} \label{Equation 4.2}\ker(\delta_z^*)=\{ 0 \}~ \mbox{and}~ \mathrm{rng}(\delta_z^*)=\{ f_Y : f \in \mathfrak{X} \ominus \mathrm{rng}(E-zI) \}.
\end{equation}
Thus, \begin{eqnarray*}\dim(\ker K_{\alpha}(z))
&=&\dim(\ker \delta_z \delta_{\alpha}^*)\\
&=& \dim(\ker \delta_{\alpha}^*)+\dim(\ker \delta_z \cap \mathrm{rng}~\delta_{\alpha}^*)\\
&=& \dim(\ker \delta_z \cap \mathrm{rng}~\delta_{\alpha}^*)\\
&=& \dim(\mathfrak{M}_z \cap \mathfrak{M}_{\alpha}^{\perp}).\end{eqnarray*}
Similarly, $$\dim(\ker K_{\alpha}(z))^*=\dim(\ker K_{z}(\alpha))=\dim(\mathfrak{M}_{\alpha} \cap \mathfrak{M}_{z}^{\perp}).$$
By \cite[Lemma $2.1$]{straus}, it is known that \begin{equation}\label{Equation 4.3}
\mathfrak{M}_{\zeta} \cap \mathfrak{M}_{\xi}^{\perp}=\{ 0 \} ~\mbox{and}~\mathfrak{M}_{\zeta} \dotplus \mathfrak{M}_{\xi}^{\perp}= \mathfrak{X} 
\end{equation} for all non real points $\zeta$, $\xi$ in the same (upper or lower) half planes. Now, if $a \in \Omega \cap \mathbb{R}$, then since $a$ is a point of regular type of $E$, the argument given in Remark \ref{Remark 3.2} implies that \begin{equation} \label{Equation 4.4}\mathfrak{M}_{\zeta} \cap \mathfrak{M}_{a}^{\perp}=\{ 0 \} ~\mbox{and}~\mathfrak{M}_{\zeta} \dotplus \mathfrak{M}_{a}^{\perp}=\mathfrak{X} ~\mbox{for all}~\zeta \in \Omega \cap(\mathbb{C}_+ \cup \mathbb{C}_-). \end{equation}
For any two closed subspaces $M$, $N$ of a Banach space, it is known that $$(M+N)^{\perp}= M^{\perp} \cap N^{\perp}.$$
By substituting $\mathfrak{M}_{\zeta}$ and $\mathfrak{M}_{a}^{\perp}$ as $M$ and $N$ respectively, we get that $\{ 0 \}= \mathfrak{X}^{\perp}= \mathfrak{M}_{\zeta}^{\perp} \cap (\mathfrak{M}_{a}^{\perp})^{\perp}$. Again, by substituting $\mathfrak{M}_{\zeta}^{\perp}$ and $\mathfrak{M}_{a}$ as $M$ and $N$ respectively, we get that $(\mathfrak{M}_{\zeta}^{\perp} +\mathfrak{M}_{a})^{\perp}= (\mathfrak{M}_{\zeta}^{\perp})^{\perp} \cap \mathfrak{M}_{a}^{\perp}=\mathfrak{M}_{\zeta} \cap \mathfrak{M}_{a}^{\perp}=\{ 0 \}$. Hence,
\begin{equation}\label{Equation 4.5}\mathfrak{M}_{a} \cap \mathfrak{M}_{\zeta}^{\perp}=\{ 0 \} ~\mbox{and}~\mathfrak{M}_{a} \dotplus \mathfrak{M}_{\zeta}^{\perp}=\mathfrak{X} ~\mbox{for all}~\zeta \in \Omega \cap(\mathbb{C}_+ \cup \mathbb{C}_-). \end{equation}
Observe that the following transformations that are obtained using \cite[Appendix $1$, Lemma $1.2$]{kreinlect} and generalized Caley transforms (see \cite[Chapter $1$, Section $2$]{kreinlect}), 
$$I+(\zeta-\xi)(E'-\zeta)^{-1}: \mathfrak{M}_{\zeta} \cap \mathfrak{M}_{\bar{\xi}}^{\perp} \rightarrow \mathfrak{M}_{\xi} \cap \mathfrak{M}_{\bar{\zeta}}^{\perp}$$ and 
$$I+(\bar{\zeta}-\xi)(E'-\bar{\zeta})^{-1}: \mathfrak{M}_{\bar{\zeta}} \cap \mathfrak{M}_{\bar{\xi}}^{\perp} \rightarrow \mathfrak{M}_{\xi} \cap \mathfrak{M}_{\zeta}^{\perp},$$ 
where $E'$ is a self adjoint extension of $E$ within $\mathfrak{X}$, are bijective for all $\xi \in \Omega \cap (\mathbb{C}_+ \cup \mathbb{C_-})$. Using the above observation and conditions $(1)$, $(2)$ of the hypothesis, it is justified that \begin{equation} \label{Equation 4.6}\dim(\mathfrak{M}_{\omega} \cap \mathfrak{M}_{\bar{\alpha}}^{\perp})< \infty ~\mbox{for all}~\omega \in \Omega \cap \mathbb{C}_+ \end{equation} and \begin{equation} \label{a.0.6}\dim(\mathfrak{M}_{\omega} \cap \mathfrak{M}_{\alpha}^{\perp})< \infty ~\mbox{for all}~\omega \in \Omega \cap \mathbb{C}_-.\end{equation}
Thus, by conditions $(1)$, $(2)$ of the hypothesis along with the observations in \eqref{Equation 4.3}, \eqref{Equation 4.4}, \eqref{Equation 4.5}, \eqref{Equation 4.6}, we get that $\dim(\ker K_{\alpha}(z))$ and $\dim(\ker K_{\alpha}(z)^*)$ are finite dimensional for all $z \in \Omega $. Now, in order to show that $K_{\alpha}(z)$ is Fredholm for all $z \in \Omega$, it is left to show that $\mathrm{rng}~K_{\alpha}(z)$ is closed for all $z \in \Omega$, or equivalently, $\mathrm{rng}~K_{z}(\alpha)$ is closed for all $z \in \Omega$. By \cite[Corollary $2.5$]{prod}, $$\mathrm{rng}~K_{z}(\alpha) ~\mbox{ is closed if and only if}~ \ker\delta_z+ \mathrm{rng}~\delta_{\alpha}^*~\mbox{ is closed}.$$ 
Since $\psi$ is a unitary map and using the closedness condition in $(1)$ of hypothesis, together with \eqref{Equation 4.1}, \eqref{Equation 4.2}, \eqref{Equation 4.3}, \eqref{Equation 4.5}, we get that $\mathrm{rng}~K_{\alpha}(z)$ is closed. Similarly, it can be shown that $K_{\bar{\alpha}}(z)$ is Fredholm for all $z \in \Omega $.\\
$\textbf{2.}$ Since $$\dim(\ker K_{\alpha}(\alpha))= \dim(\mathfrak{M}_{\alpha} \cap \mathfrak{M}_{\alpha}^{\perp})=\{  0 \}$$ and $$\dim(\ker K_{\alpha}(\alpha)^*)= \dim(\ker K_{\alpha}(\alpha))=\{  0 \} ,$$
it is clear that $ K_{\alpha}(\alpha)$ is a Fredholm operator with index zero. Hence, by part $(2)$ of Theorem \ref{Theorem 2.5}, we get that $ K_{\alpha}(\alpha)$ is invertible. Similarly, it can be shown that $ K_{\bar{\alpha}}(\bar{\alpha})$ is also invertible.\\
$\textbf{3.}$ It can be easily seen that the operator $E$ on $\mathfrak{X}$ is symmetric and it is unitarily equivalent to the multiplication operator by an independent variable on $\mathcal{N}$. For more clarification, we refer to \cite[Section $10$]{mahapatra}. Now, using \cite[Lemma $7.4$]{mahapatra}, we get that the linear transformation $T_{\alpha}$ is an isometric isomorphism. Also, $\mathcal{N}$ is invariant under $R_{\xi}$ for all $\xi \in \Omega$, explanation of which follows similarly as in Theorem \ref{Theorem 3.3}.
\end{proof}

\noindent

\textbf{Acknowledgements:} The authors gratefully acknowledge the anonymous referee for providing several valuable improvements and suggestions on the earlier version of the manuscript, which have significantly contributed to the development of the revised version.

This work is partially supported by the FIST program of the Department of Science and Technology, Government of India, Reference No. SR/FST/MS-I/2018/22(C). The research of the second author is supported by the MATRICS grant of SERB
(MTR/2023/001324).\\

\textbf{Data Availability:} No data was used for the research described in this article.\\ 

\textbf{Conflict of interest:} The authors declare that they have no conflict of interest.

\end{document}